\documentclass[a4paper, 12pt]{amsart}
\usepackage{hyperref,upref,amssymb}

\usepackage{amsthm}
\newtheorem*{Definition}{Definition}

\newtheorem*{Remark}{Remark}
\newtheorem*{Notation}{Notation}
\newtheorem*{Theorem}{Theorem}

\numberwithin{equation}{section}

\newtheorem{thm}{Theorem}[section]
\newtheorem{lemma}[thm]{Lemma}

\theoremstyle{definition}
\newtheorem{defn}[thm]{Definition}
\newtheorem{ntn}[thm]{Notation}
\theoremstyle{remark}
\newtheorem{eg}[thm]{Example}

\title[Primitive Ideal Space of $C^*(R_+)\rtimes R^\times$]{Primitive Ideal Space of $C^*(R_+)\rtimes R^\times$}

\date{19 August 2024}
\author[X. Chen]{Xiaohui Chen}
\address{Xiaohui Chen,
Department of Mathematics and Physics, North China Electric Power University, Beijing 102206, China}
\email{xiaohui20720@126.com}
\author[H. Li]{Hui Li}
\address{Hui Li,
Department of Mathematics and Physics, North China Electric Power University, Beijing 102206, China}
\email{lihui8605@hotmail.com}
\email{50902471@ncepu.edu.cn}

\thanks{The second author was supported by Fundamental Research Funds for the Central Universities (Grant No.~2023MS076).}
\thanks{The second author is the corresponding author.}

\begin{document}

\begin{abstract}
For an integral domain $R$ satisfying certain condition, we characterize the primitive ideal space and its Jacobson topology of the semigroup crossed product $C^*(R_+) \rtimes R^\times$. The main example is when $R=\mathbb{Z}[\sqrt{-3}]$.
\end{abstract}

\maketitle

\section{Introduction}

Motivated by the pioneering paper of Bost and Connes (\cite{BC95}), Cuntz in \cite{Cun08} constructed the first ring C*-algebra. Later Cuntz and Li in \cite{CL10} generalize the work of \cite{Cun08} to integral domain with finite quotients. Eventually Li in \cite{Li10} generalize the work of \cite{Cun08} to arbitrary rings. There are more than one way of studying C*-algebra associated to rings. Independent work due to Hirshberg \cite{Hir02}; Larsen-Li \cite{LL12}; Kaliszewski-Omland-Quigg \cite{KOQ14} investigates C*-algebras from $p$-adic rings. Li in \cite{Li12} defined the notion of semigroup C*-algebras and proved that the $ax+b$-semigroup C*-algebra of a ring is an extension of the ring C*-algebra. When the ring is the ring of integers of a field, Li in \cite{Li12} proved the $ax+b$-semigroup C*-algebra is isomorphic to another construction due to Cuntz, Deninger, and Laca in \cite{CDL13}. Very recent work due to Bruce-Li in \cite{BL23, BL24} and Bruce-Kubota-Takeishi in \cite{BKT24} on the algebraic dynamical systems and their associated C*-algebras solves quite a few open problems.

For an integral domain $R$, denote by $R_+$ the additive group $(R,+)$ and denote by $R^\times$ the multiplicative semigroup $(R\setminus \{0\},\cdot)$. Then there is a natural unital and injective action of $R^\times$ on $C^*(R_+)$ by multiplication, thus we obtain a semigroup crossed product $C^*(R_+) \rtimes R^\times$. In this paper we characterize the primitive ideal space and its Jacobson topology of the semigroup crossed product $C^*(R_+) \rtimes R^\times$ under certain condition. Our main example is $R=\mathbb{Z}[\sqrt{-3}]$. The semigroup crossed product $C^*(R_+) \rtimes R^\times$ is very related to the known construction. In the appendix we show that $C^*(R_+) \rtimes R^\times$ is an extension of the boundary quotient of the opposite semigroup of the $ax+b$-semigroup of the ring and that when the ring is a GCD domain, $C^*(R_+) \rtimes R^\times$ is isomorphic to the boundary quotient of the opposite semigroup of the $ax+b$-semigroup of the ring. We point out that there is only a few literature about the opposite semigroup C*-algebra of the $ax+b$-semigroup of a ring, see for example \cite{CEL13, Li16, LN16}.

\subsection*{Standing Assumptions:} \textsf{Throughout this paper, any semigroup is assumed to be discrete, countable, unital, and left cancellative; any group is assumed to be discrete and countable; any subsemigroup of a semigroup is assumed to inherit the unit of the semigroup; any ring is assumed to be countable, unital, $0 \neq 1$; any topological space is assumed to be second countable.}

\section{Laca's Dilation Theorem Revisited}

Laca in \cite{La00} proves a very important theorem which dilates a semigroup dynamical system $(A,P,\alpha)$ to a C*-dynamical system $(B,G,\beta)$ so that the semigroup crossed product $A \rtimes_{\alpha}^{e} P$ is Morita equivalent to the crossed product $B \rtimes_{\beta} G$. In this section, we revisit Laca's theorem when $A$ is a unital commutative C*-algebra.

\begin{ntn}
Let $P$ be a subsemigroup of a group $G$ satisfying $G=P^{-1}P$. For $p,q \in P$, define $p \leq q$ if $qp^{-1} \in P$. Then $\leq$ is a reflexive, transitive, and directed relation on $P$.
\end{ntn}

\begin{thm}[{cf. \cite[Theorem~2.1]{La00}}]\label{Laca dilation thm}
Let $P$ be a subsemigroup of a group $G$ satisfying $G=P^{-1}P$, let $A=C(X)$ where $X$ is a compact Hausdorff space, and let $\alpha:P \to \mathrm{End}(A)$ be a semigroup homomorphism such that $\alpha_p$ is unital and injective for all $p \in P$. Then there exists a dynamical system $(X_\infty,G,\gamma)$ ($X_\infty$ is compact Hausdorff), such that $A \rtimes_{\alpha}^{e} P$ is Morita equivalent to $C(X_\infty) \rtimes_\gamma G$.
\end{thm}
\begin{proof}
By \cite[Theorem~2.1]{La00}, there exist a C*-dynamical system $(A_\infty,G,\beta)$ such that $A \rtimes_{\alpha}^{e} P$ is Morita equivalent to $A_\infty \rtimes_\beta G$. We cite the proof of \cite[Theorem~2.1]{La00} to sketch the construction of $A_\infty$ and the definition of $\beta$: For $p \in P$, define $A_p:=A$. For $p,q \in P$ with $p \leq q$, define $\alpha_{p,q}:A_p \to A_q$ to be $\alpha_{qp^{-1}}$. Then $\{(A_p,\alpha_{p,q}):p,q \in P,p \leq q\}$ is an inductive system. Denote by $A_\infty:=\lim_{p}(A_p,\alpha_{p,q})$, denote by $\alpha^{p}:A_p \to A_\infty$ the natural unital embedding for all $p \in P$, and denote by $\beta:G \to \mathrm{Aut}(A_\infty)$ the homomorphism satisfying $\beta_{p_0}\circ \alpha^{pp_0}=\alpha^p$ for all $p_0,p\in P$.

For $p \in P$, denote by $f_p:X \to X$ the unique surjective continuous map induced from $\alpha_p$. For $p\in P$, denote by $X_p:=X$. For $p,q \in P$ with $p \leq q$, denote by $f_{q,p}:X_q \to X_p$ the unique surjective continuous map induced from $\alpha_{p,q}$. Since $\alpha_{p,q}=\alpha_{qp^{-1}},f_{q,p}=f_{qp^{-1}}$. Then $\{(X_p,f_{q,p}):p,q \in P,p \leq q\}$ is an inverse system. Denote by $X_\infty:=\{(x_p)_{p \in P} \in \prod_{p \in P}X_p:f_{q,p}(x_q)=x_p\forall p \leq q\}$, which is the inverse limit of the inverse system. \cite[Example~II.8.2.2 (i)]{BL06}, $A_\infty \cong C(X_\infty)$. For $p \in P$, denote by $f^p:X_\infty \to X_p$ the unique projection induced from $\alpha^p$. Then $f_{q,p}\circ f^q=f^p$ for all $p,q \in P,p \leq q$. For $p,p_0 \in P,f \in C(X_\infty)$, denote by $\gamma_{p_0}:X_\infty \to X_\infty$ the unique homeomorphism such that $\beta_{p_0}(f)=f\circ \gamma_{p_0}^{-1}$.

So $(X_\infty,G,\gamma)$ is a dynamical system such that $C(X_\infty) \rtimes_\gamma G \cong A_\infty \rtimes_\beta G$. Hence $A \rtimes_{\alpha}^{e} P$ is Morita equivalent to $C(X_\infty) \rtimes_\gamma G$.
\end{proof}

\begin{ntn}
We explicitly describe $X_\infty$ and the action of $G$ on $X_\infty$ given in the above theorem.

\[
X_\infty=\{(x_p)_{p \in P} \in \prod_{p \in P}X_p:f_{q,p}(x_q)=x_p,\forall p \leq q\}.
\]

For $p_0,p,q \in P$ with $q \geq p_0,p$, for $(x_p)_{p \in P} \in X_\infty$, we have

\[(p_0\cdot (x_p))(p)=x_{pp_0}, (p_0^{-1}\cdot (x_p))(p)=f_{q,p}(x_{qp_0^{-1}}).\]

Specifically, when $G$ is abelian, we have a simpler form of the group action

 \[\frac{p_0}{q_0} \cdot (x_p)=(f_{q_0}(x_{pp_0})).\]
\end{ntn}

Our goal is to apply Theorem~\ref{Laca dilation thm} to characterize the primitive ideal space of the semigroup crossed product $C^{*}(R_+)\rtimes R^{\times}$ of an integral domain. Since $R^{\times}$ is abelian, we will need the following version of Williams' theorem.

\begin{defn}
 Let $G$ be an abelian group, let $X$ be a locally compact Hausdorff space,
 and let $\alpha:G \to \mathrm{Homeo}(X)$ be a homomorphism. For $x,y\in X$, define $x \sim y$ if $\overline{G\cdot x} =\overline{G \cdot y}$. Then $\sim $ is an equivalent relation on $X$. For $x \in X$, define $[x] :=
 \overline{G\cdot x}$, which is called the \textit {quasi-orbit} of $x$. The quotient space $Q(X /G)$ by the relation $\sim$ is called the
 \textit{quasi-orbit space}. For $x \in X$, define $G_{x}:=\{ g \in G : g \cdot x = x\}$ , which is called the \textit{isotropy
 group} (or \textit{stable group}) at $x$. For $([x],\phi ), ([y],\psi ) \in Q(X/ G)\times \widehat{G}$, define $([x],\phi )\approx ([y],\psi )$
 if $[x] = [y]$ and $\phi\vert_{G_{x}}= \psi\vert_{G_{x}}$.Then $\approx$ is an equivalent relation on $Q(X /G) \times \widehat{G}$.
\end{defn}

\begin{thm}[{\cite[Theorem~1.1]{LR00}}]\label{Williams thm}
Let $G$ be an abelian group, let $X$ be a locally compact Hausdorff space, and let $\alpha : G \to \mathrm{Homeo}(X)$ be a homomorphism. Then $\mathrm{Prim}(C_{0}(X)\rtimes_\alpha G)\cong(X\times \widehat{G})/\approx$.
\end{thm}

\section{Primitive Ideal Structure of $C^*(R_+)\rtimes R^\times$}

In this section we characterize the primitive ideal space and its Jacobson topology of the semigroup crossed product $C^*(R_+) \rtimes R^\times$ under certain condition.

\begin{ntn}
Let $R$ be an integral domain. Denote by $Q$ the field of fractions of $R$, denote by $R_+$ the additive group $(R,+)$, denote by $\widehat{R_+}$ the dual group of $R_+$, denote by $R^\times$ the multiplicative semigroup $(R\setminus \{0\},\cdot)$, denote by $Q^\times$ the enveloping group $(Q\setminus \{0\},\cdot)$ of $R^\times$, denote by $\{u_r\}_{r \in R_+}$ the family of unitaries generating $C^*(R_+)$, denote by $\alpha:R^\times \to \mathrm{End}(C^*(R_+))$ the homomorphism such that $\alpha_{p}(u_r)=u_{pr}$ for all $p \in R^\times,r \in R_+$. Observe that for any $p \in R^\times, \alpha_p$ is unital and injective, and the map $f_p:\widehat{R_+} \to \widehat{R_+},\phi \mapsto \phi(p \cdot)$ is the unique surjective continuous map induced from $\alpha_p$. Denote by $X_\infty(R):=\{\phi=(\phi_p)_{p\in R^\times}\in\prod_{p\in R^\times}\widehat{R_+}:\phi_q(\frac{q}{p}\cdot )=\phi_p,\text{ whenever } p \vert q\}$. Then $\frac{p_0}{q_0}\cdot(\phi_p)=(\phi_{pp_0}(q_0\cdot))$.
\end{ntn}

\begin{lemma}\label{calculate fixed pt}
Let $R$ be an integral domain. Fix $(\phi_p)_{p\in R^\times} \in X_\infty(R)$. If $(\phi_p)_{p\in R^\times}\neq(1)_{p \in R^\times}$, then $Q^{\times}_{\phi}=\{1_R\}$. If $(\phi_p)_{p\in R^\times}= (1)_{p\in R^\times}$, then $Q^{\times}_{\phi}=Q^{\times}$.
\end{lemma}
\begin{proof}
We prove the first statement. Suppose that there exists $1 \neq \frac{p_0}{q_0} \in Q^{\times}$ such that $\frac{p_0}{q_0} \cdot \phi=\phi$, for a contradiction. Since $(\phi_p)_{p\in R^\times}\neq(1)_{p \in R^\times}$, there exists $p_1 \in R^\times$ such that $\phi_{p_1} \neq 1$. Then for any $p \in R^{\times}, \phi_p=\phi_{pp_0}(q_0 \cdot)$. Since for any $p \in R^\times$, $\phi_{pp_0}(p_0\cdot )=\phi_p$, we deduce that $\phi_{pp_0}(p_0\cdot)=\phi_{pp_0}(q_0 \cdot)$ for all $p \in R^\times$. So for any $p \in R^\times$, we have $\phi_{pp_0}((p_0-q_0)\cdot)= 1$. Hence for any $p \in R^\times, \phi_{pp_0}((p_0-q_0)p_0\cdot)= 1$. When $p=p_1(p_0-q_0)$, we get $\phi_{p_1}=\phi_{p_1(p_0-q_0)p_0}(((p_0-q_0)p_0\cdot)= 1$, which is a contradiction. Therefore $Q^{\times}_{\phi}=\{1_R\}$.

The second statement is straightforward to prove.
\end{proof}

\begin{lemma}\label{Q(X_infty(R)/Qtimes)}
Let $R$ be an integral domain. Suppose that for any $\epsilon>0$, any $(1)_{p \in R^\times}\neq (\phi_p)_{p\in R^\times} \in X_\infty(R)$, any $\pi \in \widehat{R_+}$, any $P \in R^\times$, any $r_1,r_2,\dots, r_n \in R_+$, there exist $p,q \in R^\times$ with $P \vert p$, such that $\vert\phi_{p}(qr_i)-\pi(r_i)\vert<\epsilon,i=1,2,\dots,n$. Then $Q(X_\infty(R)/Q^\times)$ consists of only two points with the only nontrivial closed subset $\{[(1)_{p \in R^\times}]\}$.
\end{lemma}
\begin{proof}
Since $\overline{Q^{\times} \cdot (1)_{p \in R^\times}}=\overline{(1)_{p \in R^\times}}=(1)_{p \in R^\times}$, we get that for any $(1)_{p \in R^\times}\neq (\phi_p)_{p\in R^\times} \in X_\infty(R), [(\phi_p)_{p\in R^\times}] \neq [(1)_{p \in R^\times}]$.

Fix $(\phi_p)_{p\in R^\times}, (\psi_p)_{p\in R^\times} \in X_\infty(R)$ such that $(\phi_p)_{p\in R^\times}, (\psi_p)_{p\in R^\times} \neq (1)_{p \in R^\times}$. We aim to show that $[(\phi_p)_{p\in R^\times}]=[(\psi_p)_{p\in R^\times}]$. It suffices to show that $(\psi_p)_{p\in R^\times} \in \overline{Q^{\times} \cdot (\phi_p)_{p\in R^\times}}$ since $(\phi_p)_{p\in R^\times} \in \overline{Q^{\times} \cdot (\psi_p)_{p\in R^\times}}$ follows the same argument. Fix $\epsilon>0$, fix $p_1,p_2,\dots,p_n \in R^\times$, and fix $r_1,r_2,\dots,r_n \in R$. By the condition posed in the lemma, there exist $p_0,q_0 \in R^{\times}$ such that $\vert \phi_{p_1p_2\cdots p_np_0}(q_0p_1\cdots p_{i-1}p_{i+1}\cdots p_nr_j)-\psi_{p_1p_2\cdots p_n}(p_1\cdots p_{i-1}p_{i+1}\cdots p_nr_j)\vert<\epsilon$, for all $1\leq i,j\leq n$. So $\vert \phi_{p_ip_0}(q_0r_j)-\psi_{p_i}(r_j)\vert<\epsilon$, for all $1\leq i,j\leq n$. Hence $(\psi_p)_{p\in R^\times} \in \overline{Q^{\times} \cdot (\phi_p)_{p\in R^\times}}$. Therefore $[(\phi_p)_{p\in R^\times}]=[(\psi_p)_{p\in R^\times}]$.

We conclude that $Q(X_\infty(R)/Q^\times)$ consists of only two points. For any $(1)_{p \in R^\times}\neq (\phi_p)_{p\in R^\times} \in X_\infty(R)$, $\overline{Q^{\times} \cdot (\phi_p)_{p\in R^\times}}=X_\infty(R) \setminus \{(1)_{p \in R^\times}\}$ is open but not closed. Finally we deduce that $\{[(1)_{p \in R^\times}]\}$ is the only nontrivial closed subset of $Q(X_\infty(R)/Q^\times)$.
\end{proof}

\begin{thm}\label{Prim(C*(R+)rtimes Rtimes)}
Let $R$ be an integral domain satisfying the condition of Lemma~\ref{Q(X_infty(R)/Qtimes)}. Take an arbitrary $(1)_{p \in R^\times} \neq (\phi_p)_{p\in R^\times} \in X_\infty(R)$. Then $\mathrm{Prim}(C^*(R_+) \rtimes R^\times) \cong  \{[(\phi_p)_{p\in R^\times}]\} \amalg \{[(1)_{p \in R^\times}]\} \times \widehat{Q^\times}$.
\end{thm}
\begin{proof}
By Theorem~\ref{Laca dilation thm}, $\mathrm{Prim}(C^*(R_+) \rtimes R^\times)$ is Morita equivalent to $C(X_\infty(R)) \rtimes Q^\times$. So $\mathrm{Prim}(C^*(R_+) \rtimes R^\times) \cong \mathrm{Prim}(C(X_\infty(R)) \rtimes Q^\times)$. By Theorem~\ref{Williams thm} and Lemma~\ref{Q(X_infty(R)/Qtimes)}, $\mathrm{Prim}(C(X_\infty(R)) \rtimes Q^\times) \cong \{[(\phi_p)_{p\in R^\times}],[(1)_{p \in R^\times}]\} \times \widehat{Q^\times} /\approx$. By Lemma~\ref{calculate fixed pt}, $Q^{\times}_{(\phi_p)_{p\in R^\times}}=\{1_R\}$ and $Q^{\times}_{(1)_{p \in R^\times}}=Q^{\times}$. So $\mathrm{Prim}(C(X_\infty(R)) \rtimes Q^\times) \cong \{[(\phi_p)_{p\in R^\times}]\} \amalg \{[(1)_{p \in R^\times}]\} \times \widehat{Q^\times}$. Hence $\mathrm{Prim}(C^*(R_+) \rtimes R^\times) \cong \{[(\phi_p)_{p\in R^\times}]\} \amalg \{[(1)_{p \in R^\times}]\} \times \widehat{Q^\times}$, and the open sets of $\mathrm{Prim}(C^*(R_+) \rtimes R^\times)$ consists of $\{[(\phi_p)_{p\in R^\times}]\} \amalg \{[(1)_{p \in R^\times}]\} \times N$, where $N$ is an open subset of $\widehat{Q^\times}$.
\end{proof}

\begin{eg}
Let $R=\mathbb{Z}$. Then $\widehat{R_+}=\mathbb{T}$. Fix $\epsilon>0$, fix $(1)_{p \in \mathbb{Z}^\times}\neq (\phi_p)_{p\in \mathbb{Z}^\times} \in X_\infty(\mathbb{Z})$, fix $\pi \in \mathbb{T}$, fix $P \in \mathbb{Z}^\times$, and fix $r_1,r_2,\dots, r_n \in \mathbb{Z}_+$. Take an arbitrary $p_0 \in \mathbb{Z}^\times$ such that $P\vert p_0$ and $\phi_{p_0}=e^{2\pi i \theta}$ for some $\theta \in (0,1)$.

Case 1. $\theta$ is rational. Then by the properties of roots of unity, there exist $m \geq 1$ and $q_0$ such that $\vert\phi_{p_0^m}^{q_0}-\pi\vert<\epsilon/\sum_{i=1}^{n}\vert r_i \vert$.

Case 2. $\theta$ is irrational. Then by the properties of irrational rotation, there exists $q_0$ such that $\vert\phi_{p_0}^{q_0}-\pi\vert<\epsilon/\sum_{i=1}^{n}\vert r_i \vert$.

Overall, there exist $p,q \in R^\times$ with $P \vert p$, such that $\vert\phi_{p}^{q}-\pi\vert<\epsilon/\sum_{i=1}^{n}\vert r_i \vert$. For $1\leq i\leq n$, we may assume that $r_i \geq 0$ and we calculate that
\begin{align*}
\vert\phi_{p}(qr_i)-\pi(r_i)\vert&=\vert\phi_{p}^{qr_i}-\pi^{r_i}\vert
\\&=\vert \phi_{p}^{q}-\pi\vert\vert\sum_{j=0}^{r_i-1}\phi_{p}^{q(r_i-1-j)}\pi^{j}\vert
\\&\leq\vert \phi_{p}^{q}-\pi\vert\sum_{j=0}^{r_i-1}\vert\phi_{p}^{q(r_i-1-j)}\pi^{j}\vert
\\&<\epsilon r_i/\sum_{i=1}^{n}\vert r_i \vert
\\&<\epsilon.
\end{align*}

So $\mathbb{Z}$ satisfies the condition of Lemma~\ref{Q(X_infty(R)/Qtimes)}.
\end{eg}

\begin{eg}
Let $R=\mathbb{Z}[\sqrt{-3}]$. Then $\mathbb{Z}[\sqrt{-3}]_+\cong \mathbb{Z}^2$ and $\widehat{\mathbb{Z}[\sqrt{-3}]_+}\cong\mathbb{T}^2$. Fix $\epsilon>0$, fix $((1,1))_{p \in R^\times}\neq ((a_p,b_p))_{p\in R^\times} \in X_\infty(\mathbb{Z}[\sqrt{-3}])$, fix $(\pi,\rho) \in \mathbb{T}^2$, fix $P \in R^\times$, and fix $r_1+s_1\sqrt{-3} ,r_2+s_2\sqrt{-3},\dots, r_n+s_n\sqrt{-3} \in \mathbb{Z}[\sqrt{-3}]_+$. Take an arbitrary $P \vert p_0 \in R^\times$ such that $(a_{p_0},b_{p_0}) \neq (1,1)$. there exist $p,q=q_1+q_2\sqrt{-3} \in R^\times$ with $P \vert p$, such that $\vert a_{p}^{q_1}b_p^{q_2}-\pi\vert,\vert a_p^{-3q_2}b_p^{q_1}-\rho\vert<\frac{\epsilon}{\sum_{i=1}^{n}\vert r_i \vert+\vert s_i \vert}$. For $1\leq i\leq n$, we may assume that $r_i \geq 0$ and we estimate that
\begin{align*}
&\vert(a_p,b_p)(q(r_i+s_i\sqrt{-3}))-(\pi,\rho)(r_i+s_i\sqrt{-3})\vert
\\&=\vert (a_{p}^{q_1}b_p^{q_2})^{r_i}(a_p^{-3q_2}b_p^{q_1})^{s_i}-\pi^{r_i}\rho^{s_i}\vert
\\&=\vert ((a_{p}^{q_1}b_p^{q_2})^{r_i}-\pi^{r_i})(a_p^{-3q_2}b_p^{q_1})^{s_i}+\pi^{r_i}((a_p^{-3q_2}b_p^{q_1})^{s_i}-\rho^{s_i})\vert
\\&\leq\vert (a_{p}^{q_1}b_p^{q_2})^{r_i}-\pi^{r_i}\vert+\vert(a_p^{-3q_2}b_p^{q_1})^{s_i}-\rho^{s_i}\vert
\\&< \frac{\epsilon\vert r_i\vert}{\sum_{i=1}^{n}\vert r_i \vert+\vert s_i \vert}+\frac{\epsilon \vert s_i \vert}{\sum_{i=1}^{n}\vert r_i \vert+\vert s_i \vert}
\\&\leq\epsilon.
\end{align*}

So $\mathbb{Z}[\sqrt{-3}]$ satisfies the condition of Lemma~\ref{Q(X_infty(R)/Qtimes)}.

Finally we conclude that any (concrete) order of number fields (for the background about number fields, one may refer to \cite{Neu99}) satisfies the condition of Lemma~\ref{Q(X_infty(R)/Qtimes)} (using a similar argument of this example).
\end{eg}

\section*{Appendix: The relationship between $C^*(R_+) \rtimes R^\times$ and semigroup C*-algebras}

In this appendix we show that $C^*(R_+) \rtimes R^\times$ is an extension of the boundary quotient of the opposite semigroup of the $ax+b$-semigroup of the ring and that when the ring is a GCD domain, $C^*(R_+) \rtimes R^\times$ is isomorphic to the boundary quotient of the opposite semigroup of the $ax+b$-semigroup of the ring.

\begin{Definition}[{\cite[Section~2]{LR96}, \cite[Definition~2.13]{Li12}}]
Let $P$ be a semigroup, let $A$ be a unital C*-algebra, and let $\alpha:P \to \mathrm{End}(A)$ be a semigroup homomorphism such that $\alpha_p$ is injective for all $p \in P$. Define the \emph{semigroup crossed product} $A \rtimes_{\alpha} P$ to be the universal unital C*-algebra generated by the image of a unital homomorphism $i_A:A \to A \rtimes_{\alpha} P$ and a semigroup homomorphism $i_P:P \to \mathrm{Isom}(A \rtimes_{\alpha} P)$ satisfying the following conditions:
\begin{enumerate}
\item $i_P(p)i_A(a)i_P(p)^*=i_A(\alpha_p(a))$, for all $p \in P,a \in A$.
\item For any unital C*-algebra $B$, any unital homomorphism $j_A:A \to B$, any semigroup homomorphism $j_P:P \to \mathrm{Isom}(B)$ satisfying $j_P(p)j_A(a)j_P(p)^*=j_A(\alpha_p(a))$, there exists a unique unital homomorphism $\Phi:A \rtimes_{\alpha} P \to B$, such that $\Phi \circ i_A=j_A$ and $\Phi \circ i_P=j_P$.
\end{enumerate}
\end{Definition}

\begin{Remark}
\begin{enumerate}
\item $i_A(1_A)=i_P(1_P)$ are both the unit of $A \rtimes_{\alpha} P$.
\item If $\alpha_p$ is unital for all $p \in P$, then for any $p \in P, i_P(p)$ is a unitary. To see this, we calculate that $i_P(p)i_P(p)^*=i_P(p)i_A(1_A)i_P(p)^*=i_A(\alpha_p(1_A))=i_A(1_A)$.
\end{enumerate}
\end{Remark}

\begin{Notation}[{\cite{BRRW14, Li12}}]
Let $P$ be a semigroup. For $p \in P$, we also denote by $p$ the left multiplication map $q \mapsto pq$. The set of \emph{constructible right ideals} is defined to be
\[
\mathcal{J}(P):=\{p_1^{-1}q_1\cdots p_n^{-1}q_nP:n \geq 1, p_1,q_1,\dots, p_n,q_n \in P\} \cup \{\emptyset\}.
\]
A finite subset $F \subset \mathcal{J}(P)$ is called a \emph{foundation set} if for any nonempty $X \in \mathcal{J}(P)$, there exists $Y \in F$ such that $X \cap Y \neq \emptyset$.
\end{Notation}

\begin{Definition}[{\cite[Remark~5.5]{BRRW14}, \cite[Definition~2.2]{Li12}}]
Let $P$ be a semigroup. Define the \emph{full semigroup C*-algebra} $C^*(P)$ of $P$ to be the universal unital C*-algebra generated by a family of isometries $\{v_p\}_{p \in P}$ and a family of projections $\{e_X\}_{X \in \mathcal{J}(P)}$ satisfying the following relations:
\begin{enumerate}
\item\label{vpvq=vpq} $v_p v_q=v_{pq}$ for all $p,q \in P$;
\item  $v_p e_X v_p^*=e_{pX}$ for all $p \in P,X \in \mathcal{J}(P)$;
\item $e_\emptyset=0$ and $e_P=1$;
\item\label{eXeY=eXY} $e_X e_Y=e_{X \cap Y}$ for all $X,Y \in \mathcal{J}(P)$.
\end{enumerate}
Define the \emph{boundary quotient} $\mathcal{Q}(P)$ of $C^*(P)$ to be the universal unital C*-algebra generated by a family of isometries $\{v_p\}_{p \in P}$ and a family of projections $\{e_X\}_{X \in \mathcal{J}(P)}$ satisfying Condition~(1)--(4) and for any foundation set $F \subset \mathcal{J}(P), \prod_{X \in F}(1-e_X)=0$.
\end{Definition}

\begin{Definition}[{\cite[Definition~2.1]{BRRW14}, \cite[Definition~2.17]{Nor14}}]
Let $P$ be a semigroup. Then $P$ is said to be \emph{right LCM} (or satisfy the \emph{Clifford condition}) if the intersection of two principal right ideals is either empty or a principal right ideal.
\end{Definition}

\begin{Notation}
Let $P$ be a semigroup. Denote by $P^{\mathrm{op}}$ the opposite semigroup of $P$. Let $R$ be an integral domain. Denote by $R_+ \rtimes R^\times$ the $ax+b$-semigroup of $R$. Denote by $\times$ the multiplication of $(R_+ \rtimes R^\times)^{\mathrm{op}}$, that is $(r_1,p_1)\times (r_2,p_2)=(r_2,p_2)(r_1,p_1)=(r_2+p_2r_1,p_1p_2)$.
\end{Notation}

\begin{Remark}
Let $R$ be an integral domain. Then any nonempty element of $\mathcal{J}((R_+ \rtimes R^\times)^{\mathrm{op}})$ is a foundation set of $(R_+ \rtimes R^\times)^{\mathrm{op}}$. To see this, for any $(r_1,p_1),(r_2,p_2) \in (R_+ \rtimes R^\times)^{\mathrm{op}}$, we compute that
\begin{align*}
(r_1,p_1) \times (p_1r_2,p_2)&=(p_1r_2,p_2)(r_1,p_1)
\\&=(p_1r_2+p_2r_1,p_1p_2)
\\&=(p_2r_1,p_1)(r_2,p_2)
\\&=(r_2,p_2)\times (p_2r_1,p_1).
\end{align*}
\end{Remark}

\begin{Theorem}
Let $R$ be an integral domain. Then $C^*(R_+) \rtimes R^\times$ is an extension of $\mathcal{Q}((R_+ \rtimes R^\times)^{\mathrm{op}})$. Moreover, if $R$ is a GCD domain (see \cite{CG00}), then $C^*(R_+) \rtimes R^\times \cong \mathcal{Q}((R_+ \rtimes R^\times)^{\mathrm{op}})$.
\end{Theorem}
\begin{proof}
Denote by $i_A:C^*(R_+) \to C^*(R_+) \rtimes R^\times$ and $i_P:R^{\times} \to \mathrm{Isom}(C^*(R_+) \rtimes R^\times)$ the canonical homomorphisms generating $C^*(R_+) \rtimes R^\times$. Denote by $\{v_{(r,p)}:(r,p) \in (R_+ \rtimes R^\times)^{\mathrm{op}}\}$ the family of isometries and by $\{e_X:X \in \mathcal{J}((R_+ \rtimes R^\times)^{\mathrm{op}})\}$ the family of projections generating $\mathcal{Q}((R_+ \rtimes R^\times)^{\mathrm{op}})$.

Firstly notice that for any $(r,p)\in (R_+ \rtimes R^\times)^{\mathrm{op}}, 1-v_{(r,p)}v_{(r,p)}^*=1-e_{(r,p) \times (R_+ \rtimes R^\times)^{\mathrm{op}}}=0$ because $\{(r,p)\times (R_+ \rtimes R^\times)^{\mathrm{op}} \}$ is a foundation set. So each $v_{(r,p)}$ is a unitary.

For $r \in R_+$, define $U_r:=v_{(r,1)}$. Since for any $r,s \in R_+$,

\[
U_r U_s=v_{(r,1)}v_{(s,1)}=v_{(s,1)(r,1)}=v_{(r+s,1)}=v_{(r,1)(s,1)}=v_{(s,1)}v_{(r,1)}=U_sU_r,
\]
we get a homomorphism $j_A:C^*(R_+) \to \mathcal{Q}((R_+ \rtimes R^\times)^{\mathrm{op}}),u_r \mapsto v_{(r,1)}$ by the universal property of $C^*(R_+)$.

For $p \in R^\times$, define $j_P(p):=v_{(0,p)}^*$. Since for any $p,q \in R^\times$, we have
\[
j_P(p)j_P(q)=v_{(0,p)}^*v_{(0,q)}^*=(v_{(0,q)}v_{(0,p)})^*=(v_{(0,p)(0,q)})^*=v_{(0,pq)}^*=j_P(pq),
\]
we deduce that $j_P:R^{\times} \to \mathrm{Isom}(\mathcal{Q}((R_+ \rtimes R^\times)^{\mathrm{op}}))$ is a semigroup homomorphism.

For any $p \in R^\times,r \in R_+$, we compute that
\[
j_P(p)j_A(u_r)j_P(p)^*=v_{(0,p)}^*v_{(r,1)}v_{(0,p)}=v_{(0,p)}^*v_{(pr,p)}=v_{(pr,1)}=j_A(u_{pr})=j_A(\alpha_p(u_r)).
\]

By the universal property of $C^*(R_+) \rtimes R^\times$, there exists a unique homomorphism $\Phi:C^*(R_+) \rtimes R^\times \to \mathcal{Q}((R_+ \rtimes R^\times)^{\mathrm{op}})$, such that $\Phi \circ i_A=j_A$ and $\Phi \circ i_P=j_P$. Since for any $(r,p) \in (R_+ \rtimes R^\times)^{\mathrm{op}}, v_{(r,p)}=v_{(0,p)}v_{(r,1)}, \Phi$ is surjective. So $C^*(R_+) \rtimes R^\times$ is an extension of $\mathcal{Q}((R_+ \rtimes R^\times)^{\mathrm{op}})$.

Now we assume that $R$ is a GCD domain. By \cite[Proposition~2.23]{Nor14}, $R^\times$ is right LCM. For any $(r_1,p_1),(r_2,p_2) \in (R_+ \rtimes R^\times)^{\mathrm{op}}$, suppose that $p_1R^\times \cap p_2R^\times=pR^\times$ for some $p \in R^\times$. We claim that $(r_1,p_1)\times (R_+ \rtimes R^\times)^{\mathrm{op}}\cap(r_2,p_2) \times (R_+ \rtimes R^\times)^{\mathrm{op}}=(0,p) \times (R_+ \rtimes R^\times)^{\mathrm{op}}$. We prove this claim. For any $(s_1,q_1),(s_2,q_2) \in (R_+ \rtimes R^\times)^{\mathrm{op}}$, if $(r_1,p_1)\times (s_1,q_1)=(r_2,p_2) \times (s_2,q_2)$, then $(r_1,p_1)\times (s_1,q_1)=(r_2,p_2) \times (s_2,q_2)=(0,p) \times (s_1+q_1r_1,\frac{q_1p_1}{p})$. Conversely, for any $(s,q) \in (R_+ \rtimes R^\times)^{\mathrm{op}}$, we have $(0,p) \times (s,q)=(r_1,p_1)\times (s-\frac{pqr_1}{p_1},\frac{pq}{p_1})=(r_2,p_2) \times (s-\frac{pqr_2}{p_2},\frac{pq}{p_2})$. So we finished proving the claim. Hence $(R_+ \rtimes R^\times)^{\mathrm{op}}$ is right LCM as well.

Since $(R_+ \rtimes R^\times)^{\mathrm{op}}$ is right LCM, by \cite[Lemma~3.4]{Sta15}, $\mathcal{Q}((R_+ \rtimes R^\times)^{\mathrm{op}})$ is the universal unital C*-algebra generated by a family of unitaries $\{v_{(r,p)}:(r,p) \in (R_+ \rtimes R^\times)^{\mathrm{op}}\}$ satisfying the following conditions:
\begin{enumerate}
\item $v_{(r_1,p_1)}v_{(r_2,p_2)}=v_{(r_1,p_1) \times (r_2,p_2)}$;
\item $v_{(r_1,p_1)}^*v_{(r_2,p_2)}=v_{(s_1,q_1)}v_{(s_2,q_2)}^*$, whenever $(r_1,p_1)\times (s_1,q_1)=(r_2,p_2) \times (s_2,q_2)$ and $(r_1,p_1)\times (R_+ \rtimes R^\times)^{\mathrm{op}}\cap(r_2,p_2)\times (R_+ \rtimes R^\times)^{\mathrm{op}}=(r_1,p_1)\times (s_1,q_1)\times (R_+ \rtimes R^\times)^{\mathrm{op}}$.
\end{enumerate}

For $(r,p) \in (R_+ \rtimes R^\times)^{\mathrm{op}}$, define $V_{(r,p)}:=i_P(p)^*i_A(u_r)$. Finally we check that $\{V_{(r,p)}\}$ satisfies the above two conditions. For any $(r_1,p_1) ,(r_2,p_2) \in (R_+ \rtimes R^\times)^{\mathrm{op}}$, we have
\begin{align*}
V_{(r_1,p_1)}V_{(r_2,p_2)}&=i_P(p_1)^*i_A(u_{r_1})i_P(p_2)^*i_A(u_{r_2})
\\&=i_P(p_1)^*i_P(p_2)^*i_A(\alpha_{p_2}(u_{r_1}))i_A(u_{r_2})
\\&=(i_P(p_2)i_P(p_1))^*i_A(u_{p_2r_1})i_A(u_{r_2})
\\&=i_P(p_1p_2)^*i_A(u_{r_2+p_2r_1})
\\&=V_{(r_2+p_2r_1,p_1p_2)}
\\&=V_{(r_1,p_1) \times (r_2,p_2)}
\end{align*}

For $(r_1,p_1),(r_2,p_2) \in (R_+ \rtimes R^\times)^{\mathrm{op}}$, suppose that $p_1R^\times \cap p_2R^\times=pR^\times$ for some $p \in R^\times$.  By the above claim $(r_1,p_1)\times (R_+ \rtimes R^\times)^{\mathrm{op}}\cap(r_2,p_2)\times (R_+ \rtimes R^\times)^{\mathrm{op}}=(0,p)\times (R_+ \rtimes R^\times)^{\mathrm{op}}$. It is not hard to see that $(r_1,p_1) \times (-\frac{pr_1}{p_1},\frac{p}{p_1})=(r_2,p_2) \times (-\frac{pr_2}{p_2},\frac{p}{p_2})=(0,p)$. So
\begin{align*}
V_{(r_1,p_1)}^*V_{(r_2,p_2)}&=i_A(u_{-r_1})i_P(p_1)i_P(p_2)^*i_A(u_{r_2})
\\&=i_A(u_{-r_1})i_P(\frac{p}{p_1})^*i_P(\frac{p}{p_2})i_A(u_{r_2})
\\&=i_P(\frac{p}{p_1})^*i_A(u_{-\frac{pr_1}{p_1}})i_A(u_{\frac{pr_2}{p_2}})i_P(\frac{p}{p_2})
\\&=V_{(-\frac{pr_1}{p_1},\frac{p}{p_1})}V_{(-\frac{pr_2}{p_2},\frac{p}{p_2})}^*.
\end{align*}

By the universal property of $\mathcal{Q}((R_+ \rtimes R^\times)^{\mathrm{op}})$, there exists a homomorphism $\Psi:\mathcal{Q}((R_+ \rtimes R^\times)^{\mathrm{op}}) \to C^*(R_+) \rtimes R^\times$ such that $\Psi(v_{(r,p)})=i_P(p)^*i_A(u_r)$.

Since
\[
\Phi \circ \Psi(v_{(r,p)})=\Phi(i_P(p)^*i_A(u_r))=j_P(p)^*j_A(u_r)=v_{(0,p)}v_{(r,1)}=v_{(r,p)},
\]
\[
\Psi\circ \Phi(i_A(u_r))=\Psi(j_A(u_r))=\Psi(v_{(r,1)})=i_A(u_r),
\]
\[
\Psi\circ \Phi(i_P(p))=\Psi(j_P(p))=\Psi(v_{(0,p)})^*=i_P(p),
\]
we conclude that $C^*(R_+) \rtimes R^\times \cong \mathcal{Q}((R_+ \rtimes R^\times)^{\mathrm{op}})$.
\end{proof}
\subsection*{Acknowledgments}

The first author wants to thank the second author for his encouragement and patient supervision.

\end{document}